\documentclass[english]{amsart}

\usepackage{amsmath,amssymb,enumerate,bbm}

\usepackage[T1]{fontenc}
\usepackage[utf8]{inputenc}
\usepackage[all]{xy}

\usepackage{babel}
\usepackage{amstext}
\usepackage{amsmath}
\usepackage{amsfonts}
\usepackage{latexsym}
\usepackage{ifthen}

\usepackage{xypic}
\xyoption{all}
\newtheorem{lemma1}[equation]{}

\newenvironment{lemma}{\begin{lemma1}{\bf Lemma.}}{\end{lemma1}}

\newenvironment{theorem}{\begin{lemma1}{\bf Theorem.}}{\end{lemma1}}
\newenvironment{proposition}{\begin{lemma1}{\bf Proposition.}}{\end{lemma1}}
\newenvironment{corollary}{\begin{lemma1}{\bf Corollary.}}{\end{lemma1}}
\newenvironment{remark}{\begin{lemma1}{\bf Remark.}\rm}{\end{lemma1}}
\newenvironment{definition}{\begin{lemma1}{\bf Definition.}}{\end{lemma1}}

\newenvironment{question}{\begin{lemma1}{\bf Question.}}{\end{lemma1}}

\newenvironment{remark*}{{\bf Remark.}}{}
\newenvironment{example*}{{\bf Example.}}{}

\makeatletter
\ifnum\@ptsize=0  \fi \ifnum\@ptsize=2  \fi 

\newcommand\Hom{{\rm Hom}}

\title{Curve Classes on Rationally Connected Varieties}
\date{25 June, 2012}

\author{Runhong Zong}
\address{Department of Mathematics, Princeton University, A5 Fine Hall Washington Road, Princeton, NJ, 08544}
\email{rzong@math.princeton.edu}
\begin{document}
\maketitle

\begin{abstract}
This note proves every curve on a rationally connected variety is algebraically equivalent to a $\mathbb{Z}$-linear combination of rational curves.

\end{abstract}

\section{Introduction}

In \cite{kollarportugaliae} and \cite{voisin}, the following question is asked by Professor J\'anos Koll\'ar and Professor Claire Voisin:
\begin{question}
For a smooth projective rationally connected variety over $\mathbb{C}$ with dimension $n$, is every integral Hodge $(n-1,n-1)$--class a $\mathbb{Z}$--linear combination of cohomology classes of rational curves?
\end{question}

This question can be separated into two questions, as in \cite{voisin}:
\begin{itemize}
\item  For a smooth projective rationally connected variety over $\mathbb{C}$, is every integral Hodge $(n-1,n-1)$--class a $\mathbb{Z}$--linear combination of cohomology classes of curves?
\item  For a smooth projective rationally connected variety over $\mathbb{C}$, is every curve class a $\mathbb{Z}$--linear combination of cohomology classes of rational curves?

\end{itemize}
While generally unknown, the dimension $3$ case of the first question is implied by the following result of Professor Claire Voisin:
\begin{theorem}
\cite{voisinuniruled} For a smooth projective 3-fold which is uniruled or Calabi-Yau, every integral Hodge $(2,2)$--class is a $\mathbb{Z}$--linear combination of cohomology classes of curves.
\end{theorem}
A restricted form of the second question can be traced back to Fano and is believed among the community (\cite{kollarportugaliae}) to be raised by Professor V.A.~Iskovskikh around 1970's.
\begin{question}\label{i}
Let $X$ be a smooth projective Fano variety over $\mathbb{C}$. Is every curve $C$ in $X$ homologous to a $\mathbb{Z}$-linear combination of homology classes of rational curves?
\end{question}

We will solve the second question in this note, resulting in the following main theorem:

\begin{theorem}\label{maintheorem}  Let $X$ be a smooth projective rationally connected variety over $\mathbb{C}$
then every curve on $X$ is algebraically equivalent to a $\mathbb{Z}$-linear combination of rational curves.
\end{theorem}

We now outline the idea of the proof. We first lift any irreducible curve $C$ in $X$ to $X \times \mathbb{P}^1$, and we regard the latter as a fibration over $\mathbb{P}^1$. By the result 
of \cite{ghs} (as we will explain later), there are very free rational curves which are horizontal with respect to the
projection to $\mathbb{P}^1$, and free rational curves supported in the fibers $X \times \{p\}$. Just add enough of these rational curves to form a comb which can be
smoothed to a new horizontal curve $\hat{C}$ with $H^1( \hat{C},{\mathcal{N}}_{\hat{C}} ) = 0$. This is ``flexible"
in the sense of \cite{ghs}. Namely, the map $$\overline{\mathcal{M}}_{g,0}( X\times \mathbb{P}^1, \beta) \to \overline{\mathcal{M}}_{g,0}(\mathbb{P}^1, d)$$ as defined in \cite{BM}  will be proper and smooth at the point represented by the curve $\hat{C}$. So we can let $\mathcal{U}$ denote the irreducible component containing this point represented by $\hat{C}$. And we have that this forgetful map will map $\mathcal{U}$ surjectively to the main component of $\overline{\mathcal{M}}_{g,0}(\mathbb{P}^1, d)$. Here $g$ (resp.~$\beta$, $d$) is the genus (resp.~cohomology class, degree over $\mathbb{P}^1$) of $\hat{C}$. This ``stabilization" map of stable map spaces will contract non-stable components which emerge after projection to $\mathbb{P}^1$. All non-stable components are either:
\begin{itemize}
\item rational curves with at most $2$ marked points,
\item (arithmetic) elliptic curves with no marked points.
\end{itemize}
 We can exclude the second case of elliptic curves by recalling that deformations of $\hat{C}$ will always be connected.

 Now the Hurwitz scheme as the main component of $\overline{\mathcal{M}}_{g,0}(\mathbb{P}^1, d)$ is irreducible, and contains points corresponding to completely degenerated covers: covers where each component is a $\mathbb{P}^1$. So just degenerating the image of $\hat{C}$ to such a cover as a sum of rational curves--by the surjectivity discussed above and what we remarked before about the non-stable components, we will get a sum of rational curves in $\overline{\mathcal{M}}_{g,0}( X\times \mathbb{P}^1, \beta)$ algebraically equivalent to $\hat{C}$. Push forward this equivalence relation back to $X$, we will get the result.

Recalling that algebraic equivalence implies cohomological equivalence and combining theorem $1.2$, we have:
\begin{corollary}  \label{corollary 3}
For a smooth projective rationally connected 3-fold, every integral Hodge $(2,2)$-class is a $\mathbb{Z}$-linear combination of cohomology classes of rational curves.
\end{corollary}

In an upcoming paper \cite{cycle} of the author with Zhiyu Tian, we will explore further application of the "trivial product" trick in the proof.

The author would like to thank Professor J\'anos Koll\'ar for his constant support and enlightening comments on this proof, to Professor Claire Voisin who pointed out that one can actually prove the rational equivalence rather than algebraic equivalence, also to Professor Burt Totaro who first introduced to the author the question for rationally connected 3-folds and pointed out several ambiguities in first editions of this note, and the most thanks should be attributed to Zhiyu Tian, who taught the author the story of \cite{ghs} and knowledge about smoothing curves and moduli space of stable maps--without his help the author would never even dreamt of getting these results.

\section{Preliminaries}

\begin{definition}
\emph{Let $X$ be smooth projective variety over $\mathbb{C}$. It is \emph{ rationally connected} if there is a rational curve passing through $2$ general points of $X$. By a \emph{free (resp.~very free) curve} in $X$ we mean a rational curve $C\subset X$ with $T_X|_{C}$ non-negative (resp.~ample). It is well-known that for $X$ to be rationally connected is equivalent to the existence of very free curves on $X$.}
\end{definition}

\begin{definition}\label{defstable}
\emph{Let $C$ be a connected nodal curve. Let $X$ be a variety. We call a map $f: C \to X$ a \emph{stable map} if every component of $C$ which is mapped to a point is either:
\begin{itemize}
\item A curve with arithmetic genus $>1$
\item A curve with arithmetic genus $1$ with at least $1$ nodal point.
\item A curve with arithmetic genus $0$ with at least $3$ nodal points.
\end{itemize}}

\end{definition}
It is well-known that for any homology class $\beta \in H_2(X,\mathbb{Z})$ there is a good compactified moduli stack $\overline{\mathcal{M}}_{g,0}(X,\beta)$ parameterizing all stable maps $f: C\to X$ with $C$ a nodal curve of genus $g$ and $f_{*}[C]=\beta$, for details see \cite{FP}.

\begin{definition}
\emph{
\label{defcomb}
Let $k$ be an arbitrary field.
A \emph{comb with  $n$ teeth} over $k$ is a projective curve
with $n+1$ irreducible components $C_0,C_1,\dots,C_n$ over $\bar k$
satisfying the following conditions:
\begin{enumerate}
\item The curve $C_0$ is defined over $k$.
\item The union $C_1\cup\dots\cup C_n$ is defined over $k$.
(Each individual curve may not be defined over $k$.)
\item The curves $C_1,\dots,C_n$ are smooth rational curves
disjoint from each other, and each of them meets $C_0$
transversely in a single smooth point of $C_0$
(which may not be defined over $k$).
\end{enumerate}
The curve $C_0$ is called the \emph{handle} of the comb,
and $C_1,\dots,C_n$ are called the \emph{teeth}.
A \emph{rational comb} is a comb whose handle is a smooth
rational curve.}

\end{definition}
\section{Proof of the Main Theorem}

Let $\mathcal{Y}$ be a smooth projective variety with a morphism $\pi: \mathcal{Y} \to \mathbb{P}^1$ whose general fibers are rationally connected.

Let $\beta \in H_2(\mathcal{Y},\mathbb{Z})$ be a class having intersection
number $d$ with a fiber of the map $\pi$.  We have then a natural
morphism as in \cite{BM}:
$$
\varphi : \overline{\mathcal{M}}_{g,0}(\mathcal{Y},\beta) \to \overline{\mathcal{M}}_{g,0}(\mathbb{P}^1,d)
$$
defined by composing a map $f : C \to \mathcal{Y}$ with $\pi$ and collapsing
components of $C$ as necessary to make the composition $\pi \circ f$ stable.

\begin{lemma}\label{contraction}
For a stable map $f : C \to \mathcal{Y}$ which is non-constant, the components that are contracted under $$
\varphi : \overline{\mathcal{M}}_{g,0}(\mathcal{Y},\beta) \to \overline{\mathcal{M}}_{g,0}(\mathbb{P}^1,d)
$$ are all rational.
\end{lemma}
\begin{proof}
By Definition \ref{defstable}, the only possible non-stable components are:

\begin{itemize}

\item Smooth rational curve with at most $2$ intersection points with other components of $C$

\item A nodal rational curve or a smooth elliptic curve which is a connected component of the curve contracted at some step.

\end{itemize}

The elliptic curve case can be excluded since $C$ will always be connected after contraction.
\end{proof}
\begin{definition}
\emph{Let $f : C
\to \mathcal{Y}$ be a stable map from a nodal curve $C$ of genus $g$ to $\mathcal{Y}$ with
class $f_*[C] = \beta$. We say that $f$ is \emph{flexible} relative to
$\pi$ if the map $\varphi : \overline{\mathcal{M}}_{g,0}(\mathcal{Y},\beta) \to \overline{\mathcal{M}}_{g,0}(\mathbb{P}^1,d)$ is dominant at
the point $[f] \in \overline{\mathcal{M}}_{g,0}(\mathcal{Y},\beta)$ and $\pi: C \to \mathbb{P}^1$ is flat.}
\end{definition}

\begin{proposition}\label{degeneration}
A  {\it flexible} curve $f : C \to \mathcal{Y}$ can be degenerated to an effective sum of rational curves in $\mathcal{Y}$.
\end{proposition}
\begin{proof}
It is a classical fact that the variety $\overline{\mathcal{M}}_{g,0}(\mathbb{P}^1,d)$ has a unique
irreducible component whose general member corresponds to a flat map
$g : C \to \mathbb{P}^1$, see \cite{F}. Since the map
$\varphi : \overline{\mathcal{M}}_{g,0}(\mathcal{Y},\beta) \to \overline{\mathcal{M}}_{g,0}(\mathbb{P}^1,d)$ is proper, and
$f : C \to \mathcal{Y}$ is {\it flexible} then $\varphi$ will be
surjective on the component of $f : C \to \mathcal{Y}$.  By Lemma \ref{contraction} it is enough to find a degeneration of $C \to \mathbb{P}^1$ in $\overline{\mathcal{M}}_{g,0}(\mathbb{P}^1,d)$ as a sum of rational curves, which is elementary.

\end{proof}

\begin{lemma} \label{cancel}
\cite{araujo} Let $X$ be a smooth projective variety
of dimension at least $3$
over an algebraically closed field.
Let   $D\subset X$ be  a smooth irreducible curve and
 $M$  a  line bundle on $D$.
Let $C\subset X$ be a very free rational curve intersecting $D$
and let
$\hat{C}$  be a family of rational curves  on $X$ parametrized by
a neighborhood of $[C]$ in $\Hom(\mathbb{P}^1,X)$.

Then there are
curves $C_1,\dots,C_p\in \hat{C}$ such that
$D^*=D\cup C_1\cup \dots \cup C_p$ is a comb and
satisfies the following conditions:
\begin{enumerate}
\item The sheaf ${\mathcal{N}}_{D^*}$ is generated by global sections.
\item $H^1(D^*,{\mathcal{N}}_{D^*}\otimes M^*)=0$, where $M^*$ is the unique
line bundle on $D^*$ that extends $M$ and has degree $0$ on the $C_i$.
\end{enumerate}

\end{lemma}

Which leads to the following:

\begin{lemma}\label{smoothing}
For any curve $C$ in a rationally connected variety $X$, for any integer $m\gg0$, let $C_1,...,C_m$ be an $m$-tuple of very free curves such that $C\cup C_1 \cup C_2...\cup C_m$ is a comb as in Definition \ref{defcomb}. If $m\gg 0$ and if $(C_1,C_2,...,C_m)$ is sufficiently general among such $m$-tuples, then there is a sub-comb $C\cup C_{i_1}\cup C_{i_2}...\cup C_{i_k}$, $k\leq m$ which can be deformed to an irreducible curve $C'$. Moreover, a general such deformation $C'$ has $H^1(C',{\mathcal{N}}_{C'})=0$, where ${\mathcal{N}}_{C'}$ is the normal bundle of $C'$.

\end{lemma}
\begin{remark}
We note that $C$ can be highly singular in $X$. But let $C'$ be the normalization of $C$. One can first embed it as $C' \to \mathbb{P}^3$. Then project a small deformation of the diagonal map $C'\to X \times \mathbb{P}^3$ to $X$. One can get a deformation of $C$ as a smooth sub--curve $C'\subset X$. Then we can apply Lemma \ref{cancel} to get Lemma \ref{smoothing}.

\end{remark}
Here we restate the \emph{First Main Construction} in the article \cite{ghs}.
\begin{theorem}\label{main} Assume there is a multisection $(B \subset \mathcal{Y}) \to \mathbb{P}^1$ which lies in the smooth locus of $\pi: \mathcal{Y} \to \mathbb{P}^1$. Then for every integer $m\gg 0$, there exist $m$-tuples of rational curves $(C_1,C_2,...,C_m)$ such that $B\cup C_1 \cup...\cup C_m$ is a comb that can be smoothed to a {\it flexible} curve of $\mathcal{Y} \to \mathbb{P}^1$.
\end{theorem}

Now we can prove our main theorem.

\begin{proof}
 Define $\mathcal{Y}$ to be $X\times \mathbb{P}^1$. For every irreducible curve $C\subset X$, lift $C$ to a curve $C'$ in $X\times{0}$, inside $\mathcal{Y}$. Since $\mathcal{Y}$ is rationally connected, we can add enough free curves of $\mathcal{Y}$ which are horizontal with respect to the projection $\mathcal{Y}\to \mathbb{P}^1$, such that the {\it comb} can be deformed by Lemma \ref{smoothing} to a multisection $M$ of the fibration $\pi: \mathcal{Y}\to \mathbb{P}^1$. Then by Theorem \ref{main}, since the trivial family has good reduction everywhere, we can add some other rational curves to $M$ to be smoothed to a {\it flexible} curve. Then by Proposition \ref{degeneration}, it can be degenerated to a sum of rational curves. So $C'$ is algebraically equivalent to an integral sum of rational curves in $\mathcal{Y}$. The rest is simply pushing forward by the projection back to $X$.

\end{proof}

\begin{remark}
As suggested by Professor J\'anos Koll\'ar, one can prove the main theorem directly on $X$, by smoothing $C\cup C_1\cup C_2 ...\cup C_m$ to a curve $C'$ with $H^1(C', {\mathcal{T}}_{X}|_{C'})=0$ and then use the same argument above for the natural forgetful map $$ \overline{\mathcal{M}}_{g,0}(X,\beta) \to \overline{\mathcal{M}}_{g,0}$$ which is again proper and surjective--this will be discussed in \cite{cycle}.
\end{remark}

\begin{remark}
As suggested by Professor Claire Voisin, based upon our result about algebraic equivalence, one can actually prove that all curves on $X$ are rationally equivalent to a $\mathbb{Z}$--linear combination of rational curves, by using a construction of Professor J\'anos Koll\'ar, see \cite{cycle} for details.
\end{remark}

\end{document}